\documentclass[10pt]{amsart}
\usepackage[T1]{fontenc}
\usepackage[utf8]{inputenc}
\usepackage[english]{babel}
\usepackage{csquotes} 
\usepackage{color} 
\usepackage[a4paper,margin=3cm]{geometry}
\allowdisplaybreaks
\usepackage[draft]{graphicx}
\usepackage[backend=biber,style=alphabetic,bibencoding=utf8,language=auto,autolang=other,giveninits=true,doi=false,isbn=false,url=false,maxnames=10,sorting=nty]{biblatex}
\usepackage{caption} 
\usepackage[hyperfootnotes=false]{hyperref} 
\hypersetup{
	colorlinks = true,
	linkcolor = {blue},
	urlcolor = {red},
	citecolor = {blue}
}
\usepackage[nameinlink,capitalise,sort]{cleveref}
\crefname{equation}{}{} 
\crefname{enumi}{}{} 

\crefname{figure}{Figure}{Figures}

\usepackage[shortlabels]{enumitem}
\newlist{thmenum}{enumerate}{1} 
\setlist[thmenum]{label={(\roman*)},
                  ref=\thetheorem-{(\roman*)}}
\crefname{thmenumi}{Theorem}{Theorem}

\usepackage{amsmath}
\usepackage{amsthm}
\usepackage{amssymb}
\usepackage{esint} 
\usepackage{mathrsfs} 
\usepackage{faktor} 
\usepackage{dsfont} 
\usepackage[dvipsnames]{xcolor}
\usepackage{array}
\usepackage{hhline}
\usepackage[normalem]{ulem}
\usepackage{enumitem} 
\usepackage{comment} 
\usepackage[toc,page]{appendix} 
\usepackage{imakeidx} 
\usepackage{xparse} 
\usepackage{mathtools}
\usepackage{fancyhdr} 
\usepackage{ifthen} 
\usepackage{forloop} 
\usepackage{xstring}
\usepackage{tikz}
\usetikzlibrary{positioning, arrows.meta}
\usetikzlibrary{babel} 
\usetikzlibrary{cd} 
\usepackage{tikz-cd} 
\usepackage{emptypage} 
\usepackage{listings} 
\usepackage{mathabx} 
\usepackage{subcaption}

\theoremstyle{plain}
\newtheorem{lemma}{Lemma}[section]
\newtheorem{proposition}[lemma]{Proposition}
\newtheorem{theorem}[lemma]{Theorem}

\newtheorem{remark}[lemma]{Remark}

\theoremstyle{definition}

\theoremstyle{remark}

\numberwithin{equation}{section}

\newcommand{\R}{\mathbb{R}}
\newcommand{\N}{\mathbb{N}}

\mathchardef\emptyset="001F
\renewcommand{\d}{\mathrm{d}}

\newcommand{\mynorm}{{\vert\kern-0.25ex\vert\kern-0.25ex\vert}}

\newcommand{\dd}{\, \mathrm{d}}
\newcommand{\pt}{\partial_t}
\newcommand{\px}{\partial_x}
\newcommand{\pxx}{\partial^2_x}

\newcommand{\eps}{\varepsilon}
\begin{document}


\title[Nonlocal conservation laws and compensated compactness]{Singular limit for a class of nonlocal conservation laws via compensated compactness}

\author[G.~M.~Coclite]{Giuseppe Maria Coclite}
\address[G.~M.~Coclite]{Politecnico di Bari, Dipartimento di Meccanica, Matematica e Management, Via E.~Orabona 4, 70125 Bari,  Italy.}
\email{giuseppemaria.coclite@poliba.it}

\author[N.~De Nitti]{Nicola De Nitti}
\address[N.~De Nitti]{Università di Pisa, Dipartimento di Matematica, Largo Bruno Pontecorvo 5, 56127 Pisa, Italy.}
\email[]{nicola.denitti@unipi.it}

\author[K.~Huang]{Kuang Huang}
\address[K.~Huang]{The Chinese University of Hong Kong, Department of Mathematics, Shatin, New Territories, Hong Kong SAR, P.\,R.~China.}
\email[]{kuanghuang@cuhk.edu.hk}



\begin{abstract}
We consider a class of nonlocal conservation laws modeling traffic flows, given by
\( \partial_t u_\varepsilon + \partial_x(V(u_\varepsilon \ast \gamma_\eps) u_\varepsilon) = 0\),  with a rescaled convolution kernel $\gamma_\eps(\cdot) \coloneqq \eps^{-1}\gamma(\cdot/\eps)$. We establish the strong  $\mathrm L^1_{\mathrm{loc}}$-convergence of weak solutions $u_\varepsilon$ toward the entropy-admissible solution of the corresponding local conservation law as the kernel $\gamma_\varepsilon$ concentrates to a Dirac delta distribution when $\varepsilon \searrow 0$. In contrast to previous literature, we obtain compactness of the family $\{u_\varepsilon \ast \gamma_\eps\}_{\varepsilon>0}$ without relying on total variation bounds or Ole\u{\i}nik-type estimates. Instead, we establish $\mathrm L^2$-type bounds on its entropy production and use the theory of compensated compactness, assuming that the initial datum merely belongs to $\mathrm L^1\cap \mathrm L^\infty$. Our results are twofold. First, we establish the nonlocal-to-local limit for the piecewise constant kernel $\gamma(\cdot) \coloneqq \mathds{1}_{[-1,0]}(\cdot)$ combined with the affine velocity function from Greenshields' traffic model. Second, we prove the limit for strictly monotone kernels along with decreasing velocity functions. These results settle a long-standing open problem concerning the nonlocal-to-local convergence for non-convex kernels.
\end{abstract}

\maketitle

\section{Introduction}
\label{sec:intro}

\subsection{Nonlocal conservation laws}

\textit{Nonlocal conservation laws} appear in a wide variety of contexts, such as traffic flow, supply chains, crowd motion, opinion dynamics, spectral behavior of large random matrices, chemical and process engineering, sedimentation, granular erosion, materials with fading memory, and conveyor belt systems (see, e.\,g., \cite{KEIMER2023}).

In this work, we focus on a nonlocal variant of the \textit{Lighthill--Whitham--Richards (LWR) macroscopic traffic flow model} \cite{MR72606,MR75522}:  the traffic density $u:\R_+\times \R \to \R$ (with $\R_+\coloneqq [0,+\infty)$) solves the Cauchy problem 
\begin{align}\label{eq:cl}
\begin{cases}
\partial_t u_\eps(t,x)  + \partial_x\big(V\big(w_\varepsilon[u_\varepsilon](t,x)\big)\,u_\eps(t,x)\big) = 0, & t>0,\ x \in \R, \\
u_\eps(0,x) = u_0(x), & x \in \R,
\end{cases}
\end{align}
with $w_\varepsilon$ given in \cref{eq:wdef}.
 
The \textit{initial datum}\footnote{~With slight abuse of notation, we use $u_0$ to denote the initial datum, where the subscript does not correspond to $\eps=0$ in $u_\eps$.} $u_0:\R \to \R$ satisfies
\begin{align}\label{ass:u0}
    u_0 \in \mathrm L^1(\R) \cap \mathrm L^\infty(\R), \qquad 0 \le u_0 \le 1,
\end{align}
where the bounds represent normalized traffic states: \(u_0 = 0\) corresponds to an empty road, and \(u_0 = 1\) to bumper-to-bumper traffic. 

The \textit{velocity function} \(V:\R \to \R_+\) is sufficiently regular and non-increasing (i.\,e., the higher the density of cars on the road,  the lower their speed), that is,  
\begin{align}\label{ass:V}
    V \in \mathrm{Lip}([0,1]), \qquad V'(\xi) \le 0 \quad \text{for all }\xi \in [0,1].
\end{align}
It depends not on the traffic density $u_\eps$, but on the \emph{nonlocal impact}  
\begin{align}\label{eq:wdef}
w_\eps[u_\varepsilon] \coloneqq \gamma_\eps \ast u_\eps,
\end{align}
where the convolution is taken with respect to the space variable,  \(\eps>0\) is the \emph{nonlocal horizon parameter} acting as a characteristic length scale, and \(\gamma_\eps(\cdot) \coloneqq \eps^{-1}\gamma(\cdot/\eps)\) is a rescaled \emph{nonlocal kernel} with $\gamma$ satisfying
\begin{align}\label{ass:gamma}
\begin{aligned}
&\gamma \in \mathrm{BV}(\R), \qquad \|\gamma\|_{\mathrm L^1(\R)} = 1, \qquad \operatorname{supp}\gamma \subset ]-\infty,0], \\  
&\gamma \ge 0, \qquad  \gamma \ \text{non-decreasing in } ]-\infty,0].
\end{aligned}
\end{align}
For traffic flow modeling, it is natural to assume that \(\gamma\) is anisotropic and, in particular, supported in \(]-\infty,0]\) and non-decreasing: drivers adjust their speed based solely on the downstream traffic density (i.\,e., by looking forward) and give greater weight to regions closer to their current position.

Convolving \cref{eq:cl} with $\gamma_\varepsilon$ yields that the nonlocal impact $w_\eps$ (where we use the simplified notation $w_\eps \coloneqq w_\eps[u_\eps]$) satisfies the following evolution equation (in the strong sense): 
\begin{align} \label{eq:W}
  \pt w_\eps + \partial_x (V(w_\eps)w_\eps) =  \px \big(V(w_\eps)w_\eps - (V(w_\eps) u_\eps) \ast \gamma_\eps\big),
\end{align}
i.\,e., a conservation law with a \emph{local flux} and a \emph{nonlocal source} (in divergence form) that acts as a regularization term. 

Under the assumptions \crefrange{ass:u0}{ass:gamma}, there exists a unique distributional solution $u_{\varepsilon} \in \mathrm L^\infty\left(\mathbb{R}_{+} \times \mathbb{R}\right)$ of the Cauchy problem \cref{eq:cl} that satisfies the uniform $\mathrm L^\infty$-bound: 
\begin{align} \label{eq:linfty}
0 \leq u_{\varepsilon} \leq 1, \qquad 0 \leq w_{\varepsilon} \leq 1
\end{align}
and, additionally $u_{\varepsilon} \in \mathrm C^0\left(\mathbb{R}_{+}, \mathrm L_{\text {loc }}^1(\mathbb{R})\right)$. Furthermore, additional regularity of the initial datum is propagated: namely, if $u_0 \in \operatorname{Lip}(\mathbb{R})$, then $u_{\varepsilon} \in \operatorname{Lip}\left(\mathbb{R}_{+} \times \mathbb{R}\right)$. We refer, e.\,g., to \cite[Proposition 2.1]{2206.03949} (and the references cited therein, particularly \cite{zbMATH06756308,bressan2020traffic,zbMATH07615111}) for a proof of these claims.

\subsection{Nonlocal-to-local singular limit problem}

We are interested in the problem of the convergence of the families $\{u_\varepsilon\}_{\varepsilon >0}$ and $\{w_\varepsilon\}_{\varepsilon >0}$ to the (unique\footnote{~We refer to \cite{MR1304494,MR3443431,MR4789921} for the well-posedness of entropy solutions of \cref{eq:cl-l}. 
}) \textit{entropy solution} $u$ of the local conservation law 
\begin{align}\label{eq:cl-l}
\begin{cases}
\partial_t u(t,x)  + \partial_x\big(V\big(u(t,x))u(t,x)\big)   	= 0,	& t >0, \ x \in \R, \\
u(0,x) = u_0(x),	&  x\in\R.
\end{cases}
\end{align}

Owing to Banach--Alaoglu's theorem, from the $\mathrm L^\infty$-bound in \cref{eq:linfty}, we deduce that (up to subsequences) 
\begin{align} 
\label{eq:weakstar}
u_\varepsilon \overset{\ast}{\rightharpoonup} u, \
w_{\varepsilon} \overset{\ast}{\rightharpoonup} u \ \text{ weakly-$\ast$ in $\mathrm L^\infty\left(\mathbb{R}_{+} \times \mathbb{R}\right)$ as $\eps \searrow 0$,}
\end{align}
for some $u \in \mathrm L^\infty(\R_+\times\R)$. To prove that $u$ is a weak solution of the corresponding local Cauchy problem \cref{eq:cl-l},
this weak convergence is not enough because oscillations may occur; a sufficient condition, however, would be strong $\mathrm L^1_{\mathrm{loc}}$-convergence.

The classical approach to prove strong \( \mathrm L^1_{\mathrm{loc}} \)-convergence of a sequence is to prove that it is uniformly bounded in some appropriate function space of positive regularity. For example, if a uniform total variation (TV) bound holds, then strong compactness in \( \mathrm L^1_{\mathrm{loc}} \) follows from Helly's compactness theorem (see, e.\,g., \cite[Theorem 2.3, p.~14]{MR1816648}). 

The problem of establishing such TV-bounds has attracted much attention in recent years. This program was carried out in \cite[Theorem 1.1]{2206.03949}: the map $t \mapsto \mathrm{TV}(w_\eps(t,\cdot))$ was shown to be non-increasing under the additional condition\footnote{~On the other hand, \cite[Theorem 1.4]{2206.03949} shows that, without \cref{ass:convexity} (which is not entirely standard in traffic flow modeling), $t \mapsto \mathrm{TV}(w_\varepsilon(t,\cdot))$ may increase.\label{ft:convexity}} 
\begin{align}
\label{ass:convexity}
\gamma \text{ is convex in $]-\infty,0]$}
\end{align}
Therefore, if $\mathrm{TV}(u_0) < \infty$, we have that (up to subsequences)
\begin{align} \label{eq:goal}
w_{\varepsilon} \rightarrow u \quad  \text{ strongly in  $\mathrm L_{\mathrm{loc}}^1\left(\mathbb{R}_{+} \times \mathbb{R}\right)$ as $\eps \searrow 0$},
\end{align}
for some function $u \in \mathrm L^\infty(\R_+ \times \R)$ that is a weak solution of  \cref{eq:cl-l}. Moreover, as shown in \cite[Theorem 1.2]{2206.03949}, the strong convergence  \cref{eq:goal} is actually sufficient to conclude that $u$ is not only a weak solution, but the unique \emph{entropy-admissible} solution of \cref{eq:cl-l} and that the whole family $\{w_\varepsilon\}_{\varepsilon >0}$ converges as $\eps\searrow0$. 

The TV-estimate $\mathrm{TV}(w_\varepsilon(t,\cdot)) \le \mathrm{TV}(u_0)$ and the convergence result for $w_\eps$ were established earlier in \cite[Theorems 3.2 and 4.2]{MR4651679} in the particular case where $\gamma(\cdot) \coloneqq \mathds{1}_{]-\infty,0]}(\cdot)\exp(\cdot)$, leveraging the identity  
\begin{align}\label{eq:exp-identity}
\varepsilon\partial_x w_\varepsilon = w_\varepsilon - u_\varepsilon.
\end{align}
In this case, the combination of \cref{eq:exp-identity} and the \( \mathrm{TV} \)-estimate on $w_{\varepsilon}$ further allows one to deduce the convergence of the family \( \{u_\varepsilon\}_{\varepsilon > 0} \) to the same limit as \( \{w_\varepsilon\}_{\varepsilon > 0} \) (see \cite[Corollary 4.1]{MR4651679}) as $\eps\searrow0$. 
For this exponential kernel, initial data with unbounded variation can also be addressed using an Ole\u{\i}nik-type regularization effect, as demonstrated in \cite{MR4855163,MR4656976}, but for specific velocity functions. 

While the family  \( \{w_\varepsilon\}_{\varepsilon > 0} \) exhibits good stability and convergence properties, on the other hand, if the initial datum is not bounded away from zero, then, as observed in \cite{MR4300935}, establishing compactness properties of $\{u_{\varepsilon}\}_{\varepsilon >0}$ is difficult because the total variation of $u_\varepsilon$ may blow up. Nevertheless, strong convergence  were obtained in \cite{MR3944408}, for a large class of nonlocal conservation laws with monotone initial data, exploiting the fact that monotonicity is preserved throughout the evolution; and, in \cite{MR4300935}, under the assumptions that the initial datum has bounded total variation, is bounded away from zero, and satisfies a one-sided Lipschitz condition, and the kernel grows at most exponentially (that is, there exists $D>0$ such that \(\gamma(z) \le D \, \gamma'(z)\), for a.\,e.~$z \in ]-\infty, 0[$).  In \cite{MR4110434,MR4283539}, Bressan and Shen proved a convergence result for the exponential kernel $\gamma(\cdot) \coloneqq \mathds{1}_{]-\infty,0]}(\cdot)\exp(\cdot)$, provided that the initial datum is bounded away from zero and has bounded total variation, by reformulating the nonlocal conservation law as a hyperbolic system with a relaxation term. 

\subsection{Main results}

The main aim of this work is to prove \cref{eq:goal} without assuming TV-bounds on the initial data and relaxing the convexity assumption on $\gamma$. Instead, we will rely on $\mathrm L^2$-type  bounds on the \emph{entropy production}: 
\begin{align}\label{eq:goal-lemma}
\{\partial_t\eta\left(w_{\varepsilon}\right)+\partial_x q\left(w_{\varepsilon}\right)\}_{\varepsilon >0} \subset \text {compact set of } \mathrm W^{-1,2}_{\mathrm{loc}}(\R_+\times\R),
\end{align}
for any \textit{entropy} $\eta \in \mathrm C^2(\R)$ and corresponding \textit{entropy flux} $q \in \mathrm C^2(\R)$  defined by the relation $q^{\prime}(\xi)=\eta^{\prime}(\xi) \, (\xi \, V(\xi))'$. 

Our main result shows that the hypothesis \cref{eq:goal-lemma} is satisfied for some suitable choices of nonlocal kernels: the piecewise constant kernel $\gamma(\cdot) \coloneqq \mathds{1}_{[-1,0]}(\cdot)$ (if we consider the \emph{Greenshields velocity function}~\cite{greenshields1935study},  $V(\xi) \coloneqq 1 -\xi$) as well as any kernel satisfying the strict monotonicity condition $\gamma'(z)> 0$ for a.\,e.~$z <0$. 

\begin{theorem}[Compactness of the entropy production]\label{th:main}
    Let us assume that  the initial datum $u_0$ satisfies \cref{ass:u0}, the velocity function $V$ satisfies \cref{ass:V}, and the nonlocal kernel $\gamma$ satisfies \cref{ass:gamma}. Let $u_\varepsilon$ be the (unique) weak solution of \cref{eq:cl} and $w_\varepsilon \coloneqq\gamma_\varepsilon \ast u_\varepsilon$  the corresponding nonlocal impact.
 \begin{thmenum} 
    \item \label{it:constant} If we assume $\gamma(\cdot) \coloneqq \mathds{1}_{[-1,0]}(\cdot)$ and $V(\xi) \coloneqq 1 -\xi$, then $w_\varepsilon$ satisfies \cref{eq:goal-lemma} for any entropy $\eta \in \mathrm C^2(\R)$.
    \item \label{it:derivative} If we assume that
 \begin{align}\label{ass:gamma-strict}
    \gamma'(z)> 0, \qquad \text{for a.\,e. } z <0,
    \end{align}
    then $w_\varepsilon$ satisfies \cref{eq:goal-lemma}  for any entropy $\eta \in \mathrm C^2(\R)$.
    \item \label{it:exp} If, in particular, $\gamma(\cdot) \coloneqq \mathds{1}_{]-\infty,0]}(\cdot)\, \exp(\cdot)$ (which satisfies \cref{ass:gamma-strict})  and $\sup_{\xi \in [0,1]}V'(\xi)\le - v_\ast < 0$, then also $u_\varepsilon$ satisfies \cref{eq:goal-lemma}  for any entropy $\eta \in \mathrm C^2(\R)$ such that $\eta(0) = 0$.
 \end{thmenum} 
\end{theorem}

Owing to Tartar's \textit{compensated compactness theory} (which we will recall in \cref{sec:cc}), the compactness of the entropy production in \cref{eq:goal-lemma}, combined with a suitable \textit{nonlinearity} assumption of the flux $\xi \mapsto \xi \, V(\xi)$, yields \cref{eq:goal}.

\begin{theorem}[Convergence]\label{th:convergence}
        Let us assume that the initial datum $u_0$ satisfies \cref{ass:u0}, the velocity function $V$ satisfies \cref{ass:V}, the flux  
   \begin{align}\label{eq:f}
        \xi \mapsto \xi\, V(\xi) \text{ is  not affine on any non-trivial interval},
    \end{align}
and the nonlocal kernel $\gamma$ satisfies \cref{ass:gamma}. 

Let $u_\varepsilon$ be the (unique) weak solution of \cref{eq:cl} and $w_\varepsilon \coloneqq \gamma_\varepsilon \ast u_\varepsilon$ the corresponding nonlocal impact. 

 \begin{thmenum} 
    \item \label{it:conv-w} Under the assumptions of \cref{it:constant} or \cref{it:derivative}, we have that the family $\{w_\varepsilon\}_{\varepsilon >0}$ converges strongly in $\mathrm L^p_{\mathrm{loc}}(\R_+\times\R)$, for $1 \le p < \infty$, to the (unique) entropy solution $u$ of the Cauchy problem \cref{eq:cl-l}. 
\item \label{it:conv-u} Under the assumptions of \cref{it:exp}, we also have 
that the family $\{u_\varepsilon\}_{\varepsilon >0}$ converges strongly in $\mathrm L^p_{\mathrm{loc}}(\R_+\times\R)$, for $1 \le p < \infty$, to the (unique) entropy solution $u$ of the Cauchy problem \cref{eq:cl-l}.
\end{thmenum}
\end{theorem}

The novelty of \cref{th:convergence} is twofold. First, we establish the nonlocal-to-local limit for the piecewise constant kernel $\gamma(\cdot) \coloneqq \mathds{1}_{[-1,0]}(\cdot)$ combined with the affine velocity function from Greenshields' traffic model. Second, we prove the limit for strictly monotone kernels along with decreasing velocity functions. 

These results settle a long-standing open problem concerning the nonlocal-to-local convergence for non-convex kernels.

As noted later in \cref{rk:conv-u}, the assumption \(V' \le -v_\ast < 0\), which is needed in \cref{it:exp}, can be removed from \cref{it:conv-u}. Indeed, we can  bypass \cref{it:exp} and instead combine \cref{it:derivative} with an appropriate time-contraction estimate based on \cref{eq:exp-identity} to conclude.

\subsection{Outline of the paper}
In \cref{sec:cc}, we recall the main lemmas in the theory of compensated compactness. In \cref{sec:entropy}, we derive the entropy balance for $w_\varepsilon$ (and, in the case of the exponential kernel, also for $u_\varepsilon$). We then address the problem of estimating the entropy production and prove \crefrange{it:constant}{it:exp} in \crefrange{sec:constant}{sec:exp}, respectively. In \cref{sec:proof}, we put together the proof of \cref{th:main} and deduce \cref{th:convergence}. Finally, in \cref{sec:conclusion}, we conclude the paper with a summary of the main results and possible directions for future research.

\section{Compensated compactness}
\label{sec:cc}

Our starting point is the uniform $\mathrm L^\infty$-bound \cref{eq:linfty} and the subsequent weak convergence result in \cref{eq:weakstar}. We seek to apply the theory of \emph{compensated compactness} developed by Tartar \cite{MR584398,MR725524} and Murat \cite{MR506997}, which provides a description of the weak limits and in certain cases the conditions under which weak limits become strong.

From \cref{eq:weakstar}, by \cite[p.~147]{MR658130}, there exists a subsequence $\{w_{\varepsilon_k}\}_{k \in \N}$ and an associated family of probability measures, \emph{Young measures}, $\left\{\nu_{(t,x)}(\lambda): \, (t,x) \in \R_+ \times \R, \ \lambda \in\R\right\}$ such that, for any $f \in \mathrm C(\R)$, 
$$
\left(\lim _{\varepsilon_k \rightarrow 0} f\left(w_{\varepsilon_k}\right)\right)
(t,x)=\left\langle\nu_{(t,x)}, f(\lambda)\right\rangle=\int f(\lambda) \, \d \nu_{(t,x)}(\lambda), \text { for a.\,e.~$(t,x) \in \R_+ \times \R$,}
$$
where the limit is taken in the weak-$\ast$ topology of $\mathrm L^\infty(\R_+\times\R)$. The spread of the support of $\nu_{(t,x)}$ is of crucial importance: indeed, strong convergence corresponds to the statement that $\nu_{(t,x)}$ is a point mass almost everywhere (see~\cite[p.~154]{MR658130}): 
$w_\eps \to u$ strongly in $\mathrm L^p_{\mathrm{loc}}(\R_+\times\R)$, for $1 \leq p<\infty$, if and only if $\nu_{(t,x)}=\delta_{u(t,x)}$ for a.\,e.~$(t,x) \in \R_+ \times \R$.

In order to show that the  measures associated to $w_\eps$ reduce to a point mass, we rely on the ``additional structure'' provided by entropy inequalities: the mechanism of entropy dissipation
 quenches the high-frequency oscillations in the sequence $\{w_\eps\}_{\eps >0}$, thus enforcing strong convergence of the approximating solutions.

The main result that we will use can be stated in the following proposition (see, e.\,g., \cite[Chapter II, Theorem 1.1]{MR658130} and \cite[Chapter 8]{MR4789921}). 

\begin{proposition}[Compensated compactness]\label{prop:cc}
Let $f \in \mathrm{Lip}(\R)$. Let us suppose that $\left\{w_{\varepsilon}\right\}_{\eps>0}$ is a sequence of functions such that $w_{\varepsilon}   \overset{\ast}{\rightharpoonup} u$  weakly-$\ast$ in $\mathrm L^\infty\left(\R_+ \times \R\right)$ and, for all convex $\eta \in \mathrm C^2(\R)$,
\begin{align*}
\{\partial_t\eta\left(w_{\varepsilon}\right)+\partial_x q\left(w_{\varepsilon}\right)\}_{\varepsilon >0} \subset \text {compact set of } \mathrm W^{-1,2}_{\mathrm{loc}}(\R_+\times\R),
\end{align*}
where $q$ is defined by the relation $q^{\prime}(\xi)=\eta^{\prime}(\xi) f^{\prime}(\xi)$. Then
$$
f\left(w_{\varepsilon}\right)  \overset{\ast}{\rightharpoonup}  f(u), \quad \text{ weakly-$\ast$ in } \mathrm L^\infty_{\mathrm{loc}}(\R_+\times\R).
$$

If, in addition, $f$ is genuinely nonlinear---i.\,e., there is no interval on which $f$ is affine---, then 
$$
w_{\varepsilon} \rightarrow u \quad \text { strongly in } \mathrm L^p_{\mathrm{loc}}(\R_+\times\R), \ 1 \le p < +\infty
$$
(up to extracting a subsequence, if needed).
\end{proposition}

In \cite{MR584398}, moreover, it was shown that, if $f$ is strictly convex, then proving \cref{eq:cl-l} for a single $\mathrm C^2$ strictly convex entropy suffices; later, in \cite{MR1073959}, this result was refined: even if $f$ is genuinely nonlinear, but not necessarily convex, strong convergence can be obtained by using just one $\mathrm C^2$ (not necessarily convex) entropy.

A key tool often used to establish \cref{eq:goal-lemma} is Murat's compact embedding theorem (see \cite{M} or \cite[Lemma 17.2.2]{Dafermos}).

\begin{lemma}[Murat's compact embedding]  \label{lm:Murat}
Let $\Omega$ be an open bounded set of $\mathbb{R}^N$ and let $\left\{\varphi_n\right\}_{n \in \N}$ be a sequence of distributions satisfying
\begin{enumerate}[label=\arabic*.]
    \item $\left\{\varphi_n\right\}_{n \in \N}$ is a bounded subset of $\mathrm W^{-1, r}(\Omega)$ for some $r>2$; 
    \item $\varphi_n=q_n+\chi_n$, where $\left\{q_n\right\}_{n \in \N}$ is relatively compact in $\mathrm W^{-1,2}(\Omega)$ and $\left\{\chi_n\right\}_{n \in \N}$ is bounded in the space of Radon measures $\mathcal M(\Omega)$ (endowed with the norm given by the total variation of measures).
\end{enumerate} 
Then, $\left\{\varphi_n\right\}_{n \in \N}$ is relatively compact in $\mathrm W^{-1,2}(\Omega)$.
\end{lemma}

\section{Entropy balance}
\label{sec:entropy}

For a convex \textit{entropy} $\eta \in \mathrm C^2(\R)$ and the corresponding \textit{entropy flux} $q\in \mathrm C^2(\R)$ given by $q'(\xi) = \eta'(\xi) (\xi \, V(\xi))' $, we can write the \emph{entropy balance} from \cref{eq:W} as 
\begin{align}\label{eq:balance-prelim}
\pt \eta(w_\eps) + \partial_x q(w_\eps) = \eta'(w_\eps)  \px \big(V(w_\eps)w_\eps - (V(w_\eps) u_\eps) \ast \gamma_\eps\big).
\end{align}

Following \cite{DeNittiKuang2025}, we will decompose \cref{eq:balance-prelim} into a divergence form part and a remainder term with a definite sign.

\begin{proposition}[Entropy balance for $w_\varepsilon$]\label{prop:energy-g}
Let us assume that the initial data \(u_0\) satisfies \cref{ass:u0}, the velocity function \(V\) satisfies \cref{ass:V}, and the nonlocal kernel \(\gamma\) satisfies \cref{ass:gamma}. Then the following identity holds in the sense of distributions on \(\mathbb{R}_+ \times \mathbb{R}\):
\begin{align}\label{eq:entropy-balance}
\begin{aligned}
 \pt \eta(w_\eps) + \partial_x q(w_\eps) &= \px \big(\eta'(w_\eps)   \big(V(w_\eps)w_\eps - (V(w_\eps)  u_\eps) \ast \gamma_\eps\big)\big) \\ & \quad +
     \partial_x \big( \big(  u_\eps(t,\cdot) H_\eta(w_\eps(t,x)\mid w_\eps(t,\cdot)) \big) \ast\gamma_\eps \big) \\ & \quad  - \big(  u_\eps(t,\cdot) H_\eta(w_\eps(t,x)\mid w_\eps(t,\cdot)) \big) \ast\gamma'_\eps,
\end{aligned}
\end{align}
where \(\eta \in \mathrm C^2(\mathbb{R})\) and 
\begin{align}
 \label{def:Heta}   H_\eta(a\mid b) &\coloneqq  I_\eta(b) - I_\eta(a) - V(b)(\eta'(b)-\eta'(a)) \ge 0, \\
\label{def:Ieta} I_\eta'(\xi) &\coloneqq  \eta''(\xi)V(\xi).
\end{align}

Moreover, if $\eta$ is convex, the function $H_\eta$ satisfies $H_\eta(a\mid b)\geq H_\eta(b\mid b)=0$ for all $a,b\in\mathbb{R}$ and 
\begin{align} \label{eq:Heta-estimate}
    0 \leq H_\eta(a,b) \leq \|\eta''\|_{\mathrm{L}^\infty} \int_a^b V(z)-V(b)\,\d z \leq \|\eta''\|_{\mathrm{L}^\infty} \|V\|_{\mathrm{L}^\infty}.
\end{align}

In particular, integrating over $[0,T] \times \R$, if $\eta(\xi) \ge 0$ for $\xi \in [0,1]$, we have 
\begin{align}\label{eq:diss-general}
\int_0^T \int_{\mathbb R}\big(  u_\eps(t,\cdot) H_\eta(w_\eps(t,x)\mid w_\eps(t,\cdot)) \big) \ast\gamma'_\eps \, \mathrm d t \le \int_{\mathbb R} \eta(u_0) \, \mathrm d x.
\end{align}

For the quadratic entropy \(\eta(\xi) \coloneqq \frac{1}{2} \xi^2\), we have
\begin{align}\label{eq:entropy-balance-quadratic}
 \begin{aligned}
 \partial_t \left( \tfrac{1}{2} w_\varepsilon^2 \right)
 + \partial_x q(w_\varepsilon)
 &= \partial_x \left[ w_\varepsilon \left( V(w_\varepsilon) w_\varepsilon - \left( V(w_\varepsilon)  u_\varepsilon \right) \ast \gamma_\varepsilon \right) \right] \\
 &\quad + \partial_x \left[ \left(  u_\varepsilon(t,\cdot)\, \int_{w_\varepsilon(t,x)}^{w_\varepsilon(t,\cdot)} (V(z) - V(w_\varepsilon(t,\cdot)))\, \mathrm{d}z \right) \ast \gamma_\varepsilon \right] \\
 &\quad - \left(  u_\varepsilon(t,\cdot)\, \int_{w_\varepsilon(t,x)}^{w_\varepsilon(t,\cdot)} (V(z) - V(w_\varepsilon(t,\cdot)))\, \mathrm{d}z \right) \ast \gamma'_\varepsilon,
 \end{aligned}
 \end{align}
and, integrating over $[0,T] \times \R$, 
\begin{align}\label{eq:diss}
\int_{0}^T \int_{\mathbb{R}} \left(  u_\varepsilon(t,\cdot)\, \int_{w_\varepsilon(t,x)}^{w_\varepsilon(t,\cdot)} (V(z) - V(t,w_\varepsilon(t,\cdot)))\, \mathrm{d}z \right) \ast \gamma'_\varepsilon \, \mathrm  dx \, \mathrm d t 
\le C \|u_0\|_{\mathrm L^2(\mathbb{R})}^2.
 \end{align}
\end{proposition}

\begin{proof} \uline{Step 1.} \emph{Entropy balance.} Multiplying the equation \cref{eq:W} by \(\eta'(u_\varepsilon)\) and applying the chain rule (justified, for instance, by an approximation argument, as in \cite[Theorem 3.1]{MR4651679} or \cite[Lemma 3.1]{MR4855163}), we compute: 
\begin{align*}
    \pt \eta(w_\eps) + \partial_x q(w_\eps) &= \eta'(w_\eps)  \px \big(V(w_\eps)w_\eps - (V(w_\eps)  u_\eps) \ast \gamma_\eps\big)
    \\&= \px \big(\eta'(w_\eps)   \big(V(w_\eps)w_\eps - (V(w_\eps)  u_\eps) \ast \gamma_\eps\big)\big) \\ & \qquad - \eta''(w_\eps)   \px w_\eps \big(V(w_\eps)w_\eps - (V(w_\eps)  u_\eps) \ast \gamma_\eps\big)
    \\ &\eqqcolon \mathcal I_{1,\eps} + \mathcal I_{2,\eps}. 
\end{align*}
We introduce the function 
\[ I_\eta'(\xi) \coloneqq  \eta''(\xi)V(\xi)\] and rewrite $\mathcal I_{2,\eps}$:
\begin{align*}
    \mathcal I_{2,\eps} &= \partial_x \eta'(w_\eps) ((V(w_\eps)  u_\eps) \ast \gamma_\eps) - \partial_x I_\eta(w_\eps) w_\eps \\
    &= \partial_x \eta'(w_\eps) ((V(w_\eps)  u_\eps) \ast \gamma_\eps) - \partial_x I_\eta(w_\eps) ( u_\eps \ast \gamma_\eps) \\ &\eqqcolon \mathcal I_{2a,\eps}- \mathcal I_{2b,\eps}.
\end{align*}
In turn, we write 
\begin{align*}
    \mathcal I_{2a,\eps}
    &= \partial_x \big( \eta'(w_\eps) (V(w_\eps)  u_\eps) \ast \gamma_\eps - (\eta'(w_\eps) V(w_\eps)  u_\eps) \ast \gamma_\eps \big) \\
    &\quad - \big( (\eta'(w_\eps) V(w_\eps)  u_\eps) \ast \gamma'_\eps - (\eta'(w_\eps) V(w_\eps)  u_\eps) \ast \gamma'_\eps \big);
\\
    \mathcal I_{2b,\eps}&= \partial_x \big( I_\eta(w_\eps)   (u_\eps \ast \gamma_\eps) - (I_\eta(w_\eps)  u_\eps) \ast \gamma_\eps \big) \\ &\qquad - \big( I_\eta(w_\eps)  (u_\eps \ast \gamma'_\eps) - (I_\eta(w_\eps)  u_\eps) \ast \gamma'_\eps \big).
\end{align*}
By introducing function 
\[H_\eta(a\mid b) \coloneqq  I_\eta(b) - I_\eta(a) - V(b)(\eta'(b)-\eta'(a)),
\]
we can write
\begin{align*}
    \mathcal I_{2,\eps}   &= \partial_x\int_x^\infty  u_\eps(t,y) H_\eta(w_\eps(t,x)\mid w_\eps(t,y)) \gamma_\eps(x-y)\,\mathrm dy \\ &\qquad- \int_x^\infty  u_\eps(t,y) H_\eta(w_\eps(t,x)\mid w_\eps(t,y)) \gamma'_\eps(x-y)\,\mathrm dy
\\ &= \partial_x \big( \big(  u_\eps(t,\cdot) H_\eta(w_\eps(t,x)\mid w_\eps(t,\cdot)) \big) \ast\gamma_\eps \big) - \big(  u_\eps(t,\cdot)  H_\eta(w_\eps(t,x)\mid w_\eps(t,\cdot)) \big) \ast\gamma'_\eps. \\
\end{align*}
In conclusion, 
\begin{align*}
 \pt \eta(w_\eps) + \partial_x q(w_\eps) &= \px \big(\eta'(w_\eps)   \big(V(w_\eps)w_\eps - (V(w_\eps)  u_\eps) \ast \gamma_\eps\big)\big) \\ & \quad +
     \partial_x \big( \big(  u_\eps(t,\cdot) H_\eta(w_\eps(t,x)\mid w_\eps(t,\cdot)) \big) \ast\gamma_\eps \big) \\ & \quad  - \big(  u_\eps(t,\cdot) H_\eta(w_\eps(t,x)\mid w_\eps(t,\cdot)) \big) \ast\gamma'_\eps.
\end{align*}
where
\begin{align*}
  H_\eta(a\mid b) &\coloneqq  I_\eta(b) - I_\eta(a) - V(b)(\eta'(b)-\eta'(a)) \ge 0, \\
 I_\eta'(\xi) &\coloneqq  \eta''(\xi)V(\xi).
\end{align*}

\uline{Step 2.} \emph{Properties of $H_\eta$.} By the fundamental theorem of calculus,
\[
I_\eta(b) - I_\eta(a) = \int_a^b \eta''(z)V(z)\,\mathrm{d}z.
\]
Hence, from \eqref{def:Heta},
\[
H_\eta(a\mid b)
= \int_a^b \eta''(z)V(z)\,\mathrm{d}z
- V(b)(\eta'(b)-\eta'(a))
= \int_a^b \eta''(z)\big(V(z)-V(b)\big)\,\mathrm{d}z.
\]
In particular, we have $H(a\mid b) \ge H(b\mid b) = 0$ for all $a,\, b \in \R$. Finally, we can estimate
\[
0 \le H_\eta(a\mid b)
\le \|\eta''\|_{\mathrm{L}^\infty}\int_a^b |V(z)-V(b)|\,\mathrm{d}z
\le \|\eta''\|_{\mathrm{L}^\infty}\|V\|_{\mathrm{L}^\infty}.
\]

\end{proof}

\begin{remark}[Entropy balance for $w_\varepsilon$ (piecewise constant kernel)]\label{rk:entropy-constant-gamma}
    In the particular case of the kernel $\gamma(\cdot) \coloneqq \mathds{1}_{[-1,0]}(\cdot)$, the entropy balance in \cref{eq:entropy-balance} reduces to  
\begin{align}\label{eq:entropy-prod-const}
\begin{aligned}
&\partial_t \eta(w_\varepsilon(t,x)) + \partial_x q(w_\varepsilon(t,x))
\\ &= \partial_x \left[ \eta'(w_\varepsilon(t,x)) 
\cdot \int_{-1}^0 \left( V(w_\varepsilon(t,x)) - V(w_\varepsilon(t,x - \varepsilon s)) \right)
 u_\varepsilon(t,x - \varepsilon s)\, \mathrm{d}s \right] \\
&\quad + \partial_x \left[ \int_{-1}^0 
 u_\varepsilon(t,x - \varepsilon s)\, H_\eta(w_\varepsilon(t,x) \mid w_\varepsilon(t,x - \varepsilon s))\, \mathrm{d}s \right] \\
&\quad - \frac{1}{\varepsilon}  u_\varepsilon(t,x + \varepsilon)\, H_\eta\left( w_\varepsilon(t,x) \mid w_\varepsilon(t,x + \varepsilon) \right).
\end{aligned}
\end{align}
If we take $\eta(\xi) = \xi^2/2$, 
\begin{align}\label{eq:entropy-prod-const-quad}
\begin{aligned}
&\partial_t \left( \tfrac{1}{2} w_\varepsilon(t,x)^2 \right)
+ \partial_x \left(q(w_\varepsilon(t,x)) \right)
\\ &= \partial_x \left[ w_\varepsilon(t,x)\, 
\int_{-1}^0 \left( V(w_\varepsilon(t,x)) - V(w_\varepsilon(t,x - \varepsilon s)) \right)
 u_\varepsilon(t,x - \varepsilon s)\, \mathrm{d}s \right] \\
&\quad + \partial_x \left[ \int_{-1}^0 
 u_\varepsilon(t,x - \varepsilon s) \left( \int_{w_\varepsilon(t,x)}^{w_\varepsilon(t,x - \varepsilon s)} \left( V(z) - V(w_\varepsilon(t,x - \varepsilon s)) \right)\, \mathrm{d}z \right) \mathrm{d}s \right] \\
&\quad - \frac{1}{\varepsilon}\,  u_\varepsilon(t,x + \varepsilon)\, \int_{w_\varepsilon(t,x)}^{w_\varepsilon(t,x + \varepsilon)} \left( V(z) - V(w_\varepsilon(t,x + \varepsilon)) \right)\, \mathrm{d}z
\end{aligned}
\end{align}
and, integrating over $[0,T] \times \R$, 
\begin{align}\label{eq:diss-bound-const}
   0 \le \int_{0}^T \int_{\mathbb R}\frac{1}{\varepsilon}\,  u_\varepsilon(t,x + \varepsilon)\, \int_{w_\varepsilon(t,x)}^{w_\varepsilon(t,x + \varepsilon)} \left( V(z) - V(w_\varepsilon(t,x + \varepsilon)) \right)\, \mathrm{d}z \, \mathrm{d}x \, \mathrm d t \le C_0. 
\end{align}

Finally, in the case of the velocity function $V(\xi)\coloneqq1-\xi$, we have 
\begin{align*}
    \partial_t \left(\frac12u_\eps^2\right) = -u_\eps \partial_x(u_\eps(1-w_\eps)) = -\partial_x\left(\frac12u_\eps^2(1-w_\eps)\right) + \frac12u_\eps^2 \partial_xw_\eps,
\end{align*}
where $\partial_xw_\eps = \frac1\eps(u_\eps(\cdot+\eps)-u_\eps)$, which implies
\[ \int_\R \frac12u_0(x)^2 \,\d x - \int_\R \frac12u_\eps(T,x)^2 \,\d x = \int_0^T \int_\R \frac1{2\eps} u_\eps(t,x)^2 (u_\eps(t,x) - u_\eps(t,x+\eps)) \,\d x, \]
and thus
\begin{align}\label{eq:ueps_L2}
    \int_0^T \int_\R u_\eps(t,x)^2 (u_\eps(t,x) - u_\eps(t,x+\eps)) \,\d x \,\d t \leq \eps \|u_0\|_{\mathrm{L}^2(\R)}^2 \eqqcolon C_0 \eps,
\end{align}
for any $T>0$. We remark that, since $u_\eps$ may have discontinuities, these derivations hold, up to an approximation argument by adding a viscosity term or by smoothing the initial data (as in \cite[Theorem 3.1]{MR4651679} or \cite[Lemma 3.1]{MR4855163})).

\end{remark}

In case \(\gamma(\cdot) \coloneqq \mathds{1}_{]-\infty,0]}(\cdot)\, \exp(\cdot)\), we can rely on the identity \cref{eq:exp-identity} to deduce a suitable entropy balance for $u_\varepsilon$ as well.

\begin{proposition}[Entropy balance for $u_\varepsilon$ (exponential kernel)]\label{prop:energy-exp}
Let us assume that the initial data \(u_0\) satisfies \cref{ass:u0}, the velocity function \(V\) satisfies \cref{ass:V}, and the nonlocal kernel \(\gamma(\cdot) \coloneqq \mathds{1}_{]-\infty,0]}(\cdot)\, \exp(\cdot)\). Then the following identity holds in the sense of distributions on \(\mathbb{R}_+ \times \mathbb{R}\):
\begin{align}\label{eq:entropy-exp1}
\begin{aligned}
\partial_t \eta(u_\varepsilon) + \partial_x q(u_\varepsilon)
&= \partial_x \Big( \eta(u_\varepsilon) (V(u_\varepsilon) - V(w_\varepsilon)) + Q(u_\varepsilon) - Q(w_\varepsilon) \Big)\\ 
&\qquad + \frac{1}{\varepsilon} V'(w_\varepsilon)\big(P(w_\varepsilon) - P(u_\varepsilon)\big)(w_\varepsilon - u_\varepsilon),
\end{aligned}
\end{align}
where \(\eta \in \mathrm C^2(\mathbb{R})\), \(q'(\xi) = (\xi V(\xi))' \eta'(\xi)\), and
\[
P(\xi) \coloneqq \xi\,\eta'(\xi) - \eta(\xi), 
\qquad Q'(\xi) \coloneqq P(\xi) V'(\xi).
\]

In particular, integrating over \(\R_+ \times \mathbb{R}\), we deduce
\begin{align}\label{eq:diss-exp}
0 \le -\int_0^T\int_{\mathbb{R}}  \frac{1}{\varepsilon} V'(w_\varepsilon)\big(P(w_\varepsilon) - P(u_\varepsilon)\big)(w_\varepsilon - u_\varepsilon) \, \mathrm dx \, \mathrm d t
\le C \int_{\R} \eta(u_0) \, \mathrm d x.
\end{align}

For the quadratic entropy \(\eta(\xi) \coloneqq \frac{1}{2} \xi^2\), we have
 \begin{align}\label{eq:entropy-balance-exp1}
 \begin{aligned}
 \partial_t \left( \frac{u_\varepsilon^2}{2} \right) + \partial_x q(u_\varepsilon) 
 &= \partial_x \left( \frac{u_\varepsilon^2}{2} (V(u_\varepsilon) - V(w_\varepsilon)) + Q(u_\varepsilon) - Q(w_\varepsilon) \right)
 \\ & \qquad + \frac{1}{2\varepsilon} V'(w_\varepsilon) (w_\varepsilon - u_\varepsilon)^2 (w_\varepsilon + u_\varepsilon).
 \end{aligned}
 \end{align}
and, in particular, 
\begin{align}\label{eq:diss-quadratic-exp}
0 \le -\int_0^T \int_{\mathbb{R}}  \frac{1}{2\varepsilon} V'(w_\varepsilon) (w_\varepsilon - u_\varepsilon)^2 (w_\varepsilon + u_\varepsilon) \, \mathrm dx \, \mathrm d t 
\le \int_{\R} u_0^2 \, \mathrm d x.
\end{align}

\end{proposition}

\begin{proof}
We write the conservation law \cref{eq:cl} in the form:
\begin{align*}
\partial_t u_\varepsilon(t,x) + \partial_x (V(u_\varepsilon(t,x)) u_\varepsilon(t,x)) 
= \partial_x\Big(u_\varepsilon(t,x) \big(V(u_\varepsilon(t,x)) - V(w_\varepsilon(t,x))\big)\Big).
\end{align*}

Multiplying the equation by \(\eta'(u_\varepsilon)\) and applying the chain rule (justified, for instance, by an approximation argument, as in \cite[Theorem 3.1]{MR4651679} or \cite[Lemma 3.1]{MR4855163}) yields:
\begin{align*}
&\partial_t \eta(u_\varepsilon) + \partial_x q(u_\varepsilon)\\
&= \eta'(u_\varepsilon) \partial_x \big( u_\varepsilon (V(u_\varepsilon) - V(w_\varepsilon)) \big) \\
&= \eta'(u_\varepsilon) \partial_x u_\varepsilon (V(u_\varepsilon) - V(w_\varepsilon)) + \eta'(u_\varepsilon) u_\varepsilon \partial_x (V(u_\varepsilon) - V(w_\varepsilon)) \\
&= \partial_x \eta(u_\varepsilon) (V(u_\varepsilon) - V(w_\varepsilon)) + \eta'(u_\varepsilon) u_\varepsilon \partial_x (V(u_\varepsilon) - V(w_\varepsilon)) \\
&= \partial_x \big( \eta(u_\varepsilon)(V(u_\varepsilon) - V(w_\varepsilon)) \big) + \big( \eta'(u_\varepsilon) u_\varepsilon - \eta(u_\varepsilon) \big) \partial_x (V(u_\varepsilon) - V(w_\varepsilon)) \\
\intertext{introducing \( P(\xi) \coloneqq \eta'(\xi)\xi - \eta(\xi) \):}
&= \partial_x \big( \eta(u_\varepsilon)(V(u_\varepsilon) - V(w_\varepsilon)) \big) 
+ \big( P(u_\varepsilon) - P(w_\varepsilon) \big) \partial_x (V(u_\varepsilon) - V(w_\varepsilon)) \\
&\quad + P(w_\varepsilon) \partial_x (V(u_\varepsilon) - V(w_\varepsilon)) \\
\intertext{introducing \( Q'(\xi) \coloneqq P(\xi) V'(\xi) \), so that \( \partial_x Q(\xi) = Q'(\xi)\partial_x \xi \):}
&= \partial_x \big( \eta(u_\varepsilon)(V(u_\varepsilon) - V(w_\varepsilon)) \big) 
+ \partial_x \big(Q(u_\varepsilon) - Q(w_\varepsilon) \big)
+ V'(w_\varepsilon) \big( P(w_\varepsilon) - P(u_\varepsilon) \big) \partial_x w_\varepsilon \\
\intertext{(using \cref{eq:exp-identity})}
&= \partial_x \big( \eta(u_\varepsilon)(V(u_\varepsilon) - V(w_\varepsilon)) + Q(u_\varepsilon) - Q(w_\varepsilon) \big)
+ \frac{1}{\varepsilon} V'(w_\varepsilon) \big( P(w_\varepsilon) - P(u_\varepsilon) \big)(w_\varepsilon - u_\varepsilon),
\end{align*}
which completes the proof.
\end{proof}

\section{Proof of \texorpdfstring{\cref{it:constant}}{constant kernel case}}
\label{sec:constant}

Starting from \cref{rk:entropy-constant-gamma}, we can prove \cref{it:constant}. We will need the following \textit{Hardy--Littlewood inequality} \cite[Lemma 1.6]{burchard2009short} (which has also been used in \cite{huang2022stability,bressan2020traffic} for establishing stability estimates or entropy inequalities for \cref{eq:cl}).

\begin{lemma}[Hardy--Littlewood's inequality]
\label{lm:hardy}
Suppose that $f$ and $g$ are nonnegative measurable functions on $\R$, vanishing at infinity. Then 
\[ \int_\R f(x)g(x) \,\d x \leq \int_\R f^\ast(x)g^\ast(x) \,\d x, \]
where $f^\ast$ and $g^\ast$ are the symmetric decreasing rearrangements of $f$ and $g$, respectively.\footnote{~We say $f^\ast$ is the \emph{symmetric decreasing rearrangement} of $f$, if $f(-x)=f(x)$ for all $x\in\R$, $f^\ast$ decreasing on $[0,\infty[$, and $\mathrm{Vol}(\{x:\ f(x)>t\})=\mathrm{Vol}(\{x:\ f^\ast(x)>t\})$ for all $t\geq0$.} In particular, if $f\coloneqq F(h)$ and $g\coloneqq h(\cdot+a)$, where $F:\R\to\R$ is strictly increasing and $a\in\R$, the above inequality yields
\[ \int_{\R} F(h(x)) (h(x)-h(x+a)) \,\d x \geq 0. \]
\end{lemma}

\begin{proof}[Proof of \cref{it:constant}] \, 

\medskip\noindent\uline{Step 1.} \emph{Decomposition of the entropy production.} We decompose the entropy production computed in \cref{eq:entropy-prod-const} as follows: 
\begin{align*}
\mathcal{I}_{1,\varepsilon}
&\coloneqq \partial_x \left[ \eta'(w_\varepsilon(t,x)) 
\cdot \int_{-1}^0 \left( V(w_\varepsilon(t,x)) - V(w_\varepsilon(t,x - \varepsilon s)) \right)
 u_\varepsilon(t,x - \varepsilon s)\, \mathrm{d}s \right], \\
\mathcal{I}_{2,\varepsilon} &\coloneqq \partial_x \left[ \int_{-1}^0 
 u_\varepsilon(t,x - \varepsilon s)\, H_\eta(w_\varepsilon(t,x) \mid w_\varepsilon(t,x - \varepsilon s))\, \mathrm{d}s \right],\\
\mathcal{I}_{3,\varepsilon}
&\coloneqq \frac{1}{\varepsilon}  u_\varepsilon(t,x + \varepsilon)\, H_\eta\left( w_\varepsilon(t,x) \mid w_\varepsilon(t,x + \varepsilon) \right).
\end{align*}

Since $0\le u_\varepsilon,\,w_\varepsilon\le1$, the family $\{\mathcal{I}_{1,\varepsilon}+\mathcal{I}_{2,\varepsilon}-\mathcal{I}_{3,\varepsilon}\}_{\varepsilon>0}$ is bounded in $\mathrm W^{-1,\infty}_{\mathrm{loc}}(\R_+\times \R)$. By \cref{lm:Murat}, to conclude \cref{eq:goal-lemma} it suffices to show
\begin{align*}
&\{\mathcal{I}_{1,\varepsilon} \}_{\varepsilon>0},\; \{\mathcal{I}_{2,\varepsilon} \}_{\varepsilon>0}\;\text{are relatively compact in }\mathrm W^{-1,2}_{\mathrm{loc}}(\R_+\times \R),\\
&
\{\mathcal{I}_{3,\varepsilon}\}_{\varepsilon>0}\;\text{is bounded in }\mathrm L^1_{\mathrm{loc}}(\R_+\times \R).
\end{align*}

\medskip\noindent\uline{Step 2.} \emph{$L^1$–bound for $\{\mathcal{I}_{3,\varepsilon}\}_{\varepsilon>0}$.} 
We write 
 \[
 I_\eta(b) - I_\eta(a) = \int_a^b I_\eta'(z)\, \mathrm{d}z = \int_a^b \eta''(z)\, V(z)\, \mathrm{d}z,
 \]
 and
 \[
 \eta'(b) - \eta'(a) = \int_a^b \eta''(z)\, \mathrm{d}z.
 \]
Plugging this into the definition of \( H_\eta \), we get 
 \begin{align*}
 H_\eta(a \mid b)
 &= \int_a^b \eta''(z)\, V(z)\, \mathrm{d}z - V(b) \int_a^b \eta''(z)\, \mathrm{d}z \\
 &= \int_a^b \eta''(z)\, (V(z) - V(b))\, \mathrm{d}z.
 \end{align*}
 Thus,
 \begin{align*}
& \int_0^T \int_{\mathbb R} \frac{1}{\varepsilon}  u_\varepsilon(t,x + \varepsilon)\, H_\eta\left( w_\varepsilon(t,x) \mid w_\varepsilon(t,x + \varepsilon) \right)\, \mathrm{d}x \, \mathrm d t
\\ &= \int_0^T\int_{\mathbb R} \frac{1}{\varepsilon}  u_\varepsilon(t,x + \varepsilon)\, 
 \int_{w_\varepsilon(t,x)}^{w_\varepsilon(t,x + \varepsilon)} (V(z) - V(w_\varepsilon(t,x + \varepsilon)))\, \eta''(z)\, \mathrm{d}z\, \mathrm{d}x \, \mathrm d t.
 \end{align*}
 Since \( \| \eta'' \|_{\mathrm L^\infty([0,1])} \leq C \), we use \cref{eq:diss} to deduce 
 \begin{align}\label{eq:diss-s2}
\int_0^T \int_{\mathbb R} \frac{1}{\varepsilon}  u_\varepsilon(t,x + \varepsilon)\, H_\eta\left( w_\varepsilon(t,x) \mid w_\varepsilon(t,x + \varepsilon) \right)\, \mathrm{d}x \, \mathrm d t
 \leq C \, C_0.
 \end{align}

\medskip\noindent\uline{Step 3.} \emph{Compactness of $\{\mathcal{I}_{2,\varepsilon}\}_{\varepsilon>0}$.} 
Owing to Jensen's inequality, we estimate 
\begin{align*}
\int_0^T &\int_{\mathbb R} \left[ \int_{-1}^0 
 u_\varepsilon(t,x - \varepsilon s)\, H_\eta\!\bigl(w_\varepsilon(t,x) \mid w_\varepsilon(t,x - \varepsilon s)\bigr)\, \mathrm{d}s \right]^2\, \mathrm{d}x \, \mathrm d t \\ 
&\le \int_0^T \int_{\mathbb R} \int_{-1}^0 u_\varepsilon(t,x - \varepsilon s)^2\, H_\eta\!\bigl(w_\varepsilon(t,x) \mid w_\varepsilon(t,x - \varepsilon s)\bigr)^{2}\, \mathrm{d}s\, \mathrm{d} x\, \mathrm d t \\
&\le C\int_0^T \int_{\mathbb R} \int_{-1}^0 u_\varepsilon(t,x - \varepsilon s)\,
\bigl| H_\eta\!\bigl(w_\varepsilon(t,x) \mid w_\varepsilon(t,x - \varepsilon s)\bigr) \bigr| \, \mathrm{d}s\, \mathrm{d} x \, \mathrm d t,
\intertext{(where we used the fact that $0 \le u_\varepsilon, \, w_\varepsilon \le 1$, $V \in \mathrm{Lip}([0,1])$, $\eta \in C^2([0,1])$)}
 & \le C \int_0^T \int_{\mathbb R} \int_{-1}^0 (-\varepsilon s) \frac{1}{(-\varepsilon s)}  u_\varepsilon(t,x - \varepsilon s)\,
\bigl| H_\eta\!\bigl(w_\varepsilon(t,x) \mid w_\varepsilon(t,x - \varepsilon s)\bigr) \bigr| \, \mathrm{d}s\, \mathrm{d} x \, \mathrm d t.   
\end{align*}

Fixing $\delta>0$, we apply \cref{lm:hardy} with $F(\xi)=\frac1\delta\xi^2-\max\{\xi,\,\delta\}+\delta$ (which is non-negative and strictly increasing), $h=u_\eps$, and $a=\eps$. Then we obtain
\begin{align*}
    \int_\R \left( \frac1\delta u_\eps(t,x)^2 - \max\{u_\eps(t,x),\,\delta\} + \delta \right) (u_\eps(t,x) - u_\eps(t,x+\eps)) \,\d x \geq 0.
\end{align*}
Using \cref{eq:ueps_L2} and $\int_\R u_\eps(t,x)\,\d x = \int_\R u_\eps(t,x+\eps)\,\d x$, we have
\begin{align*}
    \int_0^T \int_\R \max\{u_\eps(t,x),\,\delta\}\, (u_\eps(t,x) - u_\eps(t,x+\eps)) \,\d x \,\d t \leq C_0\frac\eps\delta.
\end{align*} 

Integrating in $\R$, recalling the algebraic identity \((a-b)^2 = 2a(a-b) - (a^2 - b^2)\), and using $\int_\R u_\eps(t,x)^2\,\d x = \int_\R u_\eps(t,x+\eps)^2\,\d x$, we obtain
\[
\int_\R (u_\eps(t,x) - u_\eps(t,x+\eps))^2\,\d x = 2\int_\R u_\eps(t,x)\,(u_\eps(t,x) - u_\eps(t,x+\eps))\,\d x.
\]
Then, writing 
\[
u_\eps (u_\eps - u_\eps(\cdot+\eps)) = \max\{u_\eps,\,\delta\}(u_\eps - u_\eps(\cdot+\eps)) - (\max\{u_\eps,\,\delta\} - u_\eps)(u_\eps - u_\eps(\cdot+\eps)),
\]
we deduce 
\begin{align*}
    \int_0^T \int_\R (u_\eps(t,x) - u_\eps(t,x+\eps))^2 \,\d x \,\d t &= 2 \int_0^T \int_\R u_\eps(t,x) (u_\eps(t,x) - u_\eps(t,x+\eps)) \,\d x \,\d t \\
    &\leq 2 \left ( C_0\frac\eps\delta + \delta \int_0^T 2\|u_\eps(t,\cdot)\|_{\mathrm{L}^1(\R)}\,\d t \right) \\
    &= 2 \left ( C_0\frac\eps\delta + 2\delta T\|u_0\|_{\mathrm{L}^1(\R)} \right).
\end{align*}
By taking $\delta=\sqrt{\eps}$, we deduce that
\[ \int_0^T \int_\R (u_\eps(t,x) - u_\eps(t,x+\eps))^2 \,\d x \,\d t \leq CC_0\sqrt{\eps}. \]
Then we substitute $\partial_xw_\eps = \frac1\eps(u_\eps(\cdot+\eps)-u_\eps)$ into the inequality to get
\begin{align}\label{eq:weps_L2}
    \eps^2 \int_0^T \int_\R (\partial_xw_\eps(t,x))^2 \,\d x \,\d t \leq CC_0\sqrt{\eps}.
\end{align}

Now we can proceed to estimate $\int_0^T \int_{\mathbb R} \int_{-1}^0 u_\varepsilon(t,x - \varepsilon s)\,
\bigl| H_\eta\!\bigl(w_\varepsilon(t,x) \mid w_\varepsilon(t,x - \varepsilon s)\bigr) \bigr| \, \mathrm{d}s\, \mathrm{d} x \, \mathrm d t$, which is controlled by
\begin{align*}
    \int_0^T &\int_\R \int_{-1}^0 u_\eps(t,x-\eps s) (w_\eps(t,x)-w_\eps(t,x-\eps s))^2 \,\d s\,\d x \,\d t \\
    \leq& \int_0^T \int_\R \int_{-1}^0 \left( \int_0^{-\eps s} \partial_xw_\eps(t,x+y)\,\d y\right)^2 \,\d s\,\d x \,\d t \\
    \intertext{using Jensen's inequality} 
    \leq& \int_0^T \int_\R \int_{-1}^0 (-\eps s)  \int_0^{-\eps s} (\partial_xw_\eps(t,x+y))^2\,\d y \,\d s\,\d x \,\d t \\
    =& \left( \int_{-1}^0 (-\eps s) \int_0^{-\eps s} \,\d y \,\d s \right) \int_0^T \int_\R (\partial_xw_\eps(t,x))^2 \,\d x \,\d t \\
    =& \frac13 \eps^2 \int_0^T \int_\R (\partial_xw_\eps(t,x))^2 \,\d x \,\d t \\
    \leq& CC_0\sqrt{\eps},
\end{align*}
converging to zero as $\eps\searrow 0$, ensuring the compactness of $\{\mathcal{I}_{2,\eps}\}_{\eps>0}$ in $\mathrm W^{-1,2}_{\mathrm{loc}}(\R_+\times \R)$.

\medskip\noindent\uline{Step 3.} \emph{Compactness of $\{\mathcal{I}_{1,\varepsilon}\}_{\varepsilon>0}$.}  Applying Jensen's inequality and using  the fact that $0 \le u_\varepsilon, \, w_\varepsilon \le 1$ and $\eta \in C^2([0,1])$), we deduce 
\begin{align*}
&\int_0^T \int_{\mathbb R}\left[ \eta'(w_\varepsilon(t,x)) 
\int_{-1}^0 \left( V(w_\varepsilon(t,x)) - V(w_\varepsilon(t,x - \varepsilon s)) \right)
 u_\varepsilon(t,x - \varepsilon s)\, \mathrm{d}s \right]^2\, \mathrm{d}x \, \mathrm d t \\
&\leq C\int_0^T  \int_{\mathbb R}\left[\int_{-1}^0|V(w_\varepsilon(t,x))-V(w_\varepsilon(t,x - \varepsilon s))|\,u_\varepsilon(t,x-\varepsilon s)\,\mathrm{d}s\right]^2\mathrm{d}x \, \mathrm d t.
\\ &\leq C \int_0^T  \int_{\mathbb R}\int_{-1}^0 u_\varepsilon(t,x-\varepsilon s)^2\,(V(w_\varepsilon(t,x))-V(w_\varepsilon(t,x - \varepsilon s)))^2\,\mathrm{d}s\,\mathrm{d}x \, \mathrm d t.
\end{align*}
Making the change of variables \(y=x-\varepsilon s\), hence \(x=y+\varepsilon s\), we rewrite the right-hand side above as 
\begin{align}\label{eq:pause1}
C \int_0^T\int_{-1}^0\int_{\mathbb R} u_{\varepsilon}(t,y)^2\,(V(w_\varepsilon(t,y+\varepsilon s))-V(w_\varepsilon(t,y)))^2\,\mathrm{d}y\,\mathrm{d}s \, \mathrm d t.
\end{align}

We now claim that the following inequality holds:
\begin{align}\label{eq:ineqV}
\big(V(w_\varepsilon(t,y+\varepsilon s)) - V(w_\varepsilon(t,y))\big)^2
\leq 2\,\|V'\|_{\mathrm L^\infty([0,1])}
\int_{w_\varepsilon(t,y)}^{w_\varepsilon(t,y+\varepsilon s)}
\big(V(z) - V(w_\varepsilon(t,y+\varepsilon s))\big)\,\mathrm{d}z.
\end{align}
Without loss of generality, assume \(w_\varepsilon(t,y) \leq w_\varepsilon(t,y+\varepsilon s)\). By the fundamental theorem of calculus and Fubini's theorem, the right-hand side of \cref{eq:ineqV} can be rewritten as
\begin{align*}
\int_{w_\varepsilon(t,y)}^{w_\varepsilon(t,y+\varepsilon s)} 
\big(V(z) - V(w_\varepsilon(t,y+\varepsilon s))\big)\,\mathrm{d}z
&= \int_{w_\varepsilon(t,y)}^{w_\varepsilon(t,y+\varepsilon s)} 
\int_{z}^{w_\varepsilon(t,y+\varepsilon s)} -V'(r)\,\mathrm{d}r\,\mathrm{d}z \\
&= \int_{w_\varepsilon(t,y)}^{w_\varepsilon(t,y+\varepsilon s)} 
\big(r - w_\varepsilon(t,y)\big)\,(-V'(r))\,\mathrm{d}r \\
&= \int_0^{w_\varepsilon(t,y+\varepsilon s) - w_\varepsilon(t,y)} 
z\,(-V'(w_\varepsilon(t,y) + z))\,\mathrm{d}z.
\end{align*}
By the \textit{bathtub principle}\footnote{~Let us recall the bathtub principle, which is contained in the following lemma: Let \(L>0\), \(M>0\), and \(\psi:[0,L]\to\mathbb R\) be a non-decreasing function. Among all measurable \(h:[0,L]\to[0,M]\) with fixed mass \(\int_0^L h\,\mathrm{d}t=m\in[0,ML]\), the quantity \(\int_0^L \psi(z)\,h(z)\,\mathrm{d}z\) is minimized by the bang-bang profile \(h^\star(z)=M\,\mathds 1_{[0,\,m/M]}(z)\) and maximized by \(h_\star(z)=M\,\mathds 1_{[L-m/M,\,L]}(z)\). 

In our proof, we apply this lemma with \(\psi(z)=z\), \(L=b-a\), \(a=w_\varepsilon(t,y)\), \(b=w_\varepsilon(t,y+\varepsilon s)\),  and \(M=\|V'\|_{\mathrm L^\infty([0,1])}\) in order to obtain a lower-bound for the integral involving \(h(z)=-V'(w_\varepsilon(t,y)+z)\).}
(see \cite[Theorem 1.14, p.~28]{MR1817225}),
\[
\int_0^{w_\varepsilon(t,y+\varepsilon s) - w_\varepsilon(t,y)} 
z\,(-V'(w_\varepsilon(t,y) + z))\,\mathrm{d}z
\;\ge\;
\frac{\big(V(w_\varepsilon(t,y)) - V(w_\varepsilon(t,y+\varepsilon s))\big)^2}
{2\,\|V'\|_{\mathrm L^\infty([0,1])}}.
\]
Multiplying by \(2\,\|V'\|_{\mathrm L^\infty([0,1])}\) yields \cref{eq:ineqV}.

Returning to \cref{eq:pause1}, we now estimate (using \(0 \leq u_\varepsilon \leq 1\) and \cref{eq:ineqV})
\begin{align*}
&C\int_0^T \int_{-1}^0\int_{\mathbb R} u_\varepsilon(t,y)(V(w_\varepsilon(t,y+\varepsilon s))-V(w_\varepsilon(t,y)))^2\,\mathrm{d}y\,\mathrm{d}s\, \mathrm d t \\
&\leq2\,\|V'\|_{\mathrm L^\infty([0,1])}\int_0^T \int_{-1}^0\int_{\mathbb R}\int_{w_\varepsilon(t,y)}^{w_\varepsilon(t,y+\varepsilon s)}(V(z)-V(w_\varepsilon(t,y+\varepsilon s)))\,\mathrm{d}z\,\mathrm{d}y\,\mathrm{d}s \, \mathrm d t.
\end{align*}
We can now conclude similarly as in Step 3.

\medskip\noindent\uline{Step 5.} \emph{Conclusion.}
From Steps 2--4, $\{\mathcal{I}_{1,\varepsilon}\}_{\varepsilon >0}$ and $\{\mathcal{I}_{2,\varepsilon}\}_{\varepsilon >0}$ are relatively compact in $\mathrm W^{-1,2}_{\mathrm{loc}}(\R_+\times \R)$ and $\{\mathcal{I}_{3,\varepsilon}\}_{\varepsilon >0}$ is bounded in $\mathrm{L}^1_{\mathrm{loc}}(\R_+\times \R)$. By \cref{lm:Murat}, \cref{eq:goal-lemma} follows, completing the proof.

\end{proof}

\section{Proof of \texorpdfstring{\cref{it:derivative}}{strictly increasing kernel case}}
\label{sec:super}

We address the case of the super-exponential kernel.

\begin{proof}[Proof of \cref{it:derivative}] \, 

\medskip\noindent\uline{Step 1.} \emph{Decomposition of the entropy production.} We split the entropy production computed in \cref{eq:entropy-balance} as follows: 
\begin{align*}
\mathcal{I}_{1,\varepsilon}
&\coloneqq \px \big(\eta'(w_\eps)   \big(V(w_\eps)w_\eps - (V(w_\eps)  u_\eps) \ast \gamma_\eps\big)\big), \\ 
\mathcal{I}_{2,\varepsilon} &\coloneqq
     \partial_x \big( \big(  u_\eps(t,\cdot) H_\eta(w_\eps(t,x)\mid w_\eps(t,\cdot)) \big) \ast\gamma_\eps \big),\\
\mathcal{I}_{3,\varepsilon}
&\coloneqq - \big(  u_\eps(t,\cdot) H_\eta(w_\eps(t,x)\mid w_\eps(t,\cdot)) \big) \ast\gamma'_\eps.
\end{align*}

Since $0\le u_\varepsilon,\,w_\varepsilon\le1$, the family $\{\mathcal{I}_{1,\varepsilon}+\mathcal{I}_{2,\varepsilon}+\mathcal{I}_{3,\varepsilon}\}_{\varepsilon>0}$ is bounded in $\mathrm W^{-1,\infty}_{\mathrm{loc}}(\R_+\times \R)$. By \cref{lm:Murat}, to conclude \cref{eq:goal-lemma} it suffices to show
\begin{align*}
&\{\mathcal{I}_{1,\varepsilon} \}_{\varepsilon>0},\; \{\mathcal{I}_{2,\varepsilon} \}_{\varepsilon>0}\;\text{are relatively compact in }\mathrm W^{-1,2}_{\mathrm{loc}}(\R_+\times \R),\\
&
\{\mathcal{I}_{3,\varepsilon}\}_{\varepsilon>0}\;\text{is bounded in }\mathrm L^1_{\mathrm{loc}}(\R_+\times \R).
\end{align*}

\medskip\noindent\uline{Step 2.} \emph{$L^1$-bound for $\{\mathcal{I}_{3,\varepsilon}\}_{\varepsilon>0}$.}
Since, by the fundamental theorem of calculus, 
\begin{align*}
I_\eta(b) - I_\eta(a) &= \int_a^b I_\eta'(z)\, \mathrm{d}z = \int_a^b \eta''(z)\, V(z)\, \mathrm{d}z,\\
\eta'(b) - \eta'(a) &= \int_a^b \eta''(z)\, \mathrm{d}z.
\end{align*}
we deduce 
\begin{align*}
H_\eta(a \mid b) & \coloneqq I_\eta(b) - I_\eta(a) - V(b)\left( \eta'(b) - \eta'(a) \right) \\ &= \int_a^b \eta''(z)\, V(z)\, \mathrm{d}z - V(b) \int_a^b \eta''(z)\, \mathrm{d}z \\
&= \int_a^b \left( V(z) - V(b) \right)\, \eta''(z)\, \mathrm{d}z.
\end{align*}

Therefore, since $0 \le u_\varepsilon,\, w_\varepsilon \le 1$ and $\eta \in \mathrm C^2([0,1])$,
\[
\int_0^T \int_{\mathbb R}
\left(  u_\varepsilon(t,\cdot)\, H_\eta(w_\varepsilon(t,x) \mid w_\varepsilon(t,\cdot)) \right) \ast \gamma'_\varepsilon(x)\, \mathrm{d}x \, \mathrm d t
\leq C\, C_0,
\]
with $C_0$ as in \cref{eq:diss}.

\medskip\noindent\uline{Step 3.} \emph{Compactness of $\{\mathcal{I}_{1,\varepsilon}\}_{\varepsilon>0}$.}   Applying Jensen's inequality and the fact that $0 \le u_\varepsilon, \, w_\varepsilon \le 1$,   $V \in \mathrm{Lip}([0,1])$, and $\eta \in C^2([0,1])$, we estimate
\begin{align*}
&\int_0^T\int_{\mathbb R} \left[ \eta^{\prime}\left(w_{\varepsilon}(t,x)\right)\, \frac{1}{\varepsilon} \int_{\mathbb{R}} u_{\varepsilon}(t,y)\left(V \left(w_{\varepsilon}(t,x)\right) - V\left(w_{\varepsilon}(t,y)\right)\right)\, \gamma\left(\frac{x-y}{\varepsilon}\right)\, \mathrm{d}y \right]^2 \mathrm{d}x \, \mathrm d t \\
&\leq C \int_0^T \int_{\mathbb R} \left[ \frac{1}{\varepsilon} \int_{\mathbb{R}} u_{\varepsilon}(t,y)\left(V \left(w_{\varepsilon}(t,x)\right) - V\left(w_{\varepsilon}(t,y)\right)\right)\, \gamma\left(\frac{x-y}{\varepsilon}\right)\, \mathrm{d}y \right]^2 \mathrm{d}x \, \mathrm d t
\\ &\leq C\int_0^T  \int_{\mathbb R} \frac{1}{\varepsilon} \int_{\mathbb{R}} \left[ u_{\varepsilon}(t,y)\left(V \left(w_{\varepsilon}(t,x)\right) - V\left(w_{\varepsilon}(t,y)\right)\right) \right]^2\, \gamma\left( \frac{x - y}{\varepsilon} \right)\, \mathrm{d}y\, \mathrm{d}x \, \mathrm d t.
\end{align*}

Let us define
\begin{align*}
    g_\eps(t,s)\coloneqq \int_0^T \int_\R u_\eps(t,x-\eps s) \int_{w_\eps(t,x)}^{w_\eps(t,x-\eps s)} \big( V(z)-V(w_\eps(t,x-\eps s))\big) \,\d z\,\d x \d t,
\end{align*}

We aim to show 
\begin{align}\label{eq:aim-g}
\int_{-\infty}^0 g_\eps(t,s)\gamma(s)\,\d s \to0 \quad \text{as } \eps\to0. 
\end{align}

First, from \cref{eq:diss}, we deduce
\[\int_{-\infty}^0 g_\eps(t,s)\gamma'(s)\,\d s \leq C_0\,\eps\quad \text{for all }\eps>0, \]
and notice that, for a fixed $T>0$, $g$ satisfies
\[ 0 \leq g_\eps(t,s) \leq T\|V\|_{\mathrm{L}^\infty} \|u_\eps\|_{\mathrm{L}^1} =  T\|V\|_{\mathrm{L}^\infty} \|u_0\|_{\mathrm{L}^1} \eqqcolon C_1 \quad \text{for all }\eps>0. \]

To prove \cref{eq:aim-g}, we consider (recalling \cref{ass:gamma-strict})
\[ A_{\delta,M} \coloneqq \{s: \ -M\leq s \leq 0 \ \text{and} \ \gamma'(s)\geq\delta \} \quad \text{for any }\delta,M>0. \]
Then, we have
\begin{align*}
    \int_{A_{\delta,M}} g_\eps(t,s)\gamma(s)\,\d s \leq \frac{\gamma(0)}{\delta} \int_{A_{\delta,M}} g_\eps(t,s)\gamma'(s)\,\d s \leq C_0 \frac{\gamma(0)}{\delta} \eps,
\end{align*}
and
\begin{align*}
    \int_{]-\infty,0]\backslash A_{\delta,M}} g_\eps(t,s)\gamma(s)\,\d s &\leq C_1 \int_{-\infty}^0 \mathds{1}_{\{s\leq-M \text{ or } \gamma'(s)<\delta\}}(s) \gamma(s)\,\d s \\
    &\eqqcolon C_1 K_\gamma(\delta,M) \to 0 \quad \text{as } \delta\to0,\,M\to\infty,
\end{align*}
where the last line is due to Lebesgue's dominated convergence theorem.
Thus, from
\begin{align*}
    \int_{-\infty}^0 g_\eps(t,s)\gamma(s)\,\d s \leq C_0 \frac{\gamma(0)}{\delta} \eps + C_1 K_\gamma(\delta,M),
\end{align*}
by passing to the limit as $\delta,\,\eps\searrow 0$, with $\varepsilon = o(\delta)$ (i.\,e., while $\varepsilon/\delta \to 0$), and $M\to+\infty$, we deduce that \cref{eq:aim-g} holds; this, in turn, yields that $\{\mathcal I_{1,\varepsilon}\}_{\varepsilon >0}$ is compact in $\mathrm W^{-1,2}_{\mathrm{loc}}(\R_+\times \R)$.

\medskip\noindent\uline{Step 4.} \emph{Compactness of $\{\mathcal{I}_{2,\varepsilon}\}_{\varepsilon>0}$.}  
Applying Jensen's inequality and using the fact that $0 \le u_\varepsilon, \, w_\varepsilon \le 1$ and $V, \eta \in \mathrm C^2([0,1])$, we estimate 
\begin{align*}
&\int_0^T \int_{\mathbb R}
\left[ \left(  u_\varepsilon(t,\cdot)\, H_\eta(w_\varepsilon(t,x) \mid w_\varepsilon(t,\cdot)) \right) \ast \gamma_\varepsilon \right]^2 \mathrm{d}x \, \mathrm d t\\ 
&\leq C \int_0^T \int_{\mathbb R} \left(  u_\varepsilon(t,\cdot)\, \int_{w_\varepsilon(t,x)}^{w_\varepsilon(t,\cdot)} (V(z) - V(w_\varepsilon(t,\cdot)))\, \mathrm{d}z \right)^2 \ast \gamma_\varepsilon(x)\, \mathrm{d}x  \, \mathrm d t \\
&\le  C \int_0^T \int_{\mathbb R} \left(  u_\varepsilon(t,\cdot)\, \int_{w_\varepsilon(t,x)}^{w_\varepsilon(t,\cdot)} (V(z) - V(w_\varepsilon(t,\cdot)))^2\, \mathrm{d}z \right) \ast \gamma_\varepsilon(x)\, \mathrm{d}x  \, \mathrm d t
\end{align*}
We can then argue as in Step 3 to conclude the proof.

\medskip\noindent\uline{Step 5.} \emph{Conclusion.}
From Steps 2--4, $\{\mathcal{I}_{1,\varepsilon}\}_{\varepsilon >0}$ and $\{\mathcal{I}_{2,\varepsilon}\}_{\varepsilon >0}$ are relatively compact in $W^{-1,2}_{\mathrm{loc}}(\R_+\times \R)$ and $\{\mathcal{I}_{3,\varepsilon}\}_{\varepsilon >0}$ is bounded in $\mathrm{L}^1_{\mathrm{loc}}(\R_+\times \R)$. By \cref{lm:Murat}, \cref{eq:goal-lemma} follows, completing the proof.

\end{proof}

\begin{remark}[Special cases]
    For specific $\gamma$, finer estimates can be obtained in Steps 3--4 from asymptotics of $K_\gamma$. For example:
    \begin{itemize}
        \item For the linear kernel $\gamma(z)\coloneqq 2(1+z)\mathds{1}_{]-\infty,0]}(z)$, we have $K_\gamma(\delta,M)=0$ for $\delta\leq2$ and $M\geq1$, so we can choose $\delta=2,M=1$ and obtain $\int_{-\infty}^0 g_\eps(t,s)\gamma(s)\,\d s\leq \frac12C_0\gamma(0)\eps$. 
        \item For the exponential kernel $\gamma(z)\coloneqq e^z\mathds{1}_{]-\infty,0]}(z)$, by taking $\delta=1$ and $M=\ln\eps$, we obtain $K_\gamma(\delta,M)=\eps$ and thus $\int_{-\infty}^0 g_\eps(t,s)\gamma(s)\,\d s\leq (C_0\gamma(0)+C_1)\eps$.
        \item More directly, if we assume that there exists $D>0$ such that
 \begin{align}\label{ass:gamma-derivative}
    \gamma'(z) \ge D\,\gamma(z) > 0, \qquad \text{for } z <0,
    \end{align}
    (including, in particular,  $\gamma(z)\coloneqq e^z\mathds{1}_{]-\infty,0]}(z)$, which satisfies \cref{ass:gamma-derivative} with equality and $D =1$), then $\int_{-\infty}^0 g_\eps(t,s)\gamma(s)\,\d s\leq D \int_{-\infty}^0 g_\eps(t,s)\gamma'(s)\,\d s \le D\, C_0\, \varepsilon.$
    \end{itemize}
\end{remark}

\section{Proof of \texorpdfstring{\cref{it:exp}}{exponential kernel case}}
\label{sec:exp}

Using \cref{prop:energy-exp}, we to prove the additional claim contained in \cref{it:exp}.

\begin{proof}[Proof of \cref{it:exp}] \,

\medskip\noindent\uline{Step 1.} \emph{Decomposition of the entropy production.} We split the entropy production computed in \cref{eq:entropy-exp1} as follows:
\begin{align*}
\mathcal{I}_{1,\varepsilon}
&\coloneqq \partial_x\Bigl[\eta(u_\varepsilon)\bigl(V(u_\varepsilon)-V(w_\varepsilon)\bigr)
+Q(u_\varepsilon)-Q(w_\varepsilon)\Bigr],\\
\mathcal{I}_{2,\varepsilon}
&\coloneqq \frac{1}{\varepsilon}V'(w_\varepsilon)\bigl(P(w_\varepsilon)-P(u_\varepsilon)\bigr)
(w_\varepsilon-u_\varepsilon).
\end{align*}
Since $0\le u_\varepsilon,\,w_\varepsilon\le1$, the family $\{\mathcal{I}_{1,\varepsilon}+\mathcal{I}_{2,\varepsilon}\}_{\varepsilon>0}$ is bounded in $\mathrm W^{-1,\infty}_{\mathrm{loc}}(\R_+\times \R)$. By \cref{lm:Murat}, to conclude \cref{eq:goal-lemma} it suffices to show
\[
\{\mathcal{I}_{1,\varepsilon}\}_{\varepsilon>0}\;\text{is relatively compact in }\mathrm W^{-1,2}_{\mathrm{loc}}(\R_+\times \R),
\qquad
\{\mathcal{I}_{2,\varepsilon}\}_{\varepsilon>0}\;\text{is bounded in }\mathrm L^1_{\mathrm{loc}}(\R_+\times \R).
\]

\medskip\noindent\uline{Step 2.} \emph{$L^1$-bound for $\{\mathcal{I}_{2,\varepsilon}\}_{\varepsilon>0}$.}
Since \(P(\xi)=\xi\,\eta'(\xi)-\eta(\xi)\) and \(P'(\xi)=\xi\,\eta''(\xi)\), by the mean value theorem there exists \(\zeta_\varepsilon\) between \(u_\varepsilon\) and \(w_\varepsilon\) such that \[P(w_\varepsilon)-P(u_\varepsilon)=\zeta_\varepsilon\,\eta''(\zeta_\varepsilon)\,(w_\varepsilon-u_\varepsilon).\]
Therefore, using the fact that $0 \le u_\varepsilon, \, w_\varepsilon \le 1$ and $\eta \in \mathrm C^2([0,1])$, we have 
\begin{align*}
&\int_0^T\int_{\R} \frac{1}{\varepsilon} V'(w_\varepsilon)\big(P(w_\varepsilon)-P(u_\varepsilon)\big)(w_\varepsilon-u_\varepsilon)\,\mathrm{d}x\,\mathrm{d}t
\\&\le \int_0^T \int_{\mathbb R}  \frac{1}{\varepsilon} V'(w_\varepsilon) \zeta_\varepsilon\,\eta''(\zeta_\varepsilon)\,(w_\varepsilon-u_\varepsilon) ^2 \,\mathrm{d}x\,\mathrm{d}t
\\ &\le C \int_0^T \int_{\mathbb R} \frac{1}{\varepsilon} V'(w_\varepsilon) (u_\varepsilon + w_\varepsilon)(w_\varepsilon-u_\varepsilon) ^2 \,\mathrm{d}x\,\mathrm{d}t
\\ &\le C\, C_0,
\end{align*}
where in the last line we recalled \cref{eq:diss-quadratic-exp}.

\medskip\noindent\uline{Step 3.} \emph{Compactness of $\{\mathcal{I}_{1,\varepsilon}\}_{\varepsilon>0}$.}
First, since $V'(\xi)\le -v_\ast<0$ on $[0,1]$, \cref{eq:diss-quadratic-exp} yields
\begin{align}\label{eq:improved-diss-exp-vast}
\int_0^T \int_\R \frac{v_\ast}{2\varepsilon}\,(u_\varepsilon-w_\varepsilon)^2\,
(u_\varepsilon+w_\varepsilon)\,\mathrm{d}x\,\mathrm{d}t \le C_0.
\end{align}
Let us assume that $\eta(0) = 0$, so that $\eta(\xi) \le C\,|\xi|$.\footnote{~Owing to the nonlocality, \cref{eq:cl} is not necessarily translation invariant; however, here we may assume $\eta(0) = 0$  because, as observed in \cite{MR1073959}, it suffices to work with the particular entropy $\eta(\xi)\coloneqq \xi\,V(\xi)$ for \cref{prop:cc} to hold.} 

By the mean value theorem, there exist points \(\xi_\varepsilon,\,\zeta_\varepsilon\) between \(u_\varepsilon\) and \(w_\varepsilon\) such that
\begin{align*}
V(u_\varepsilon)-V(w_\varepsilon) &= V'(\zeta_\varepsilon)(u_\varepsilon - w_\varepsilon),
\\
Q(u_\varepsilon) - Q(w_\varepsilon) &= \int_{w_\varepsilon}^{u_\varepsilon} P(\xi)V'(\xi)\,\mathrm d\xi = P(\xi_\varepsilon)V'(\xi_\varepsilon)(u_\varepsilon - w_\varepsilon).
\end{align*}
Hence,
\[
\bigl|\eta(u_\varepsilon)\bigl(V(u_\varepsilon)-V(w_\varepsilon)\bigr)
+Q(u_\varepsilon)-Q(w_\varepsilon)\bigr|
= |u_\varepsilon - w_\varepsilon|\,
\bigl|V'(\zeta_\varepsilon)\eta(u_\varepsilon) + V'(\xi_\varepsilon)P(\xi_\varepsilon)\bigr|.
\]
Moreover, we estimate 
\begin{align*}
|\eta(u_\varepsilon)| &\le C u_\varepsilon \le C(u_\varepsilon + w_\varepsilon), \\
|P(\xi_\varepsilon)| &= |\xi_\varepsilon \, \eta'(\xi_\varepsilon) - \eta(\xi_\varepsilon)| 
\le |\xi_\varepsilon\eta'(\xi_\varepsilon)| + |\eta(\xi_\varepsilon)|
\le C|\xi_\varepsilon|
\le C(u_\varepsilon + w_\varepsilon),
\end{align*}
and deduce 
\[
\bigl|\eta(u_\varepsilon)\bigl(V(u_\varepsilon)-V(w_\varepsilon)\bigr)
+Q(u_\varepsilon)-Q(w_\varepsilon)\bigr|
\le C\, \|V\|_{\mathrm{L}^\infty([0,1])} |u_\varepsilon - w_\varepsilon|
(u_\varepsilon + w_\varepsilon).
\]

Taking squares, integrating over $(0,\infty)\times\R$, and recalling \cref{eq:improved-diss-exp-vast} yield
\begin{align*}
\int_0^T\!\!\int_{\R}
\bigl|\eta(u_\varepsilon)\bigl(V(u_\varepsilon)-V(w_\varepsilon)\bigr)
+Q(u_\varepsilon)-Q(w_\varepsilon)\bigr|^2 \mathrm{d}x\,\mathrm{d}t
&\le C\int_0^T\!\!\int_{\R}
(u_\varepsilon+w_\varepsilon)\,(u_\varepsilon-w_\varepsilon)^2 \mathrm{d}x\,\mathrm{d}t \\
&\le C\,\varepsilon \longrightarrow 0 \quad \text{as }\varepsilon\searrow 0.
\end{align*}

\medskip\noindent\uline{Step 4.} \emph{Conclusion.}
From Steps 2--3, $\{\mathcal{I}_{1,\varepsilon}\}_{\varepsilon >0}$ is relatively compact in $\mathrm W^{-1,2}_{\mathrm{loc}}(\R_+\times \R)$ and $\{\mathcal{I}_{2,\varepsilon}\}_{\varepsilon >0}$ is bounded in $\mathrm{L}^1_{\mathrm{loc}}(\R_+\times \R)$. By \cref{lm:Murat}, \cref{eq:goal-lemma} follows, completing the proof.
\end{proof}

\section{Proof of \texorpdfstring{\cref{th:main}}{Main Theorem} and \texorpdfstring{\cref{th:convergence}}{Nonlocal-to-local limit}}
\label{sec:proof}

Putting together the results in \crefrange{sec:constant}{sec:exp}, we conclude the proof of \cref{th:main}. As a corollary, owing to \cref{prop:cc} and \cite[Theorem 1.2]{2206.03949}, we can also prove \cref{th:convergence}.

\begin{proof}[Proof of \cref{th:convergence}]
Under the assumptions of \cref{it:constant} or \cref{it:derivative}, we can apply \cref{prop:cc} to the family $\{w_\varepsilon\}_{\varepsilon >0}$ and obtain that (up to subsequences)
\[
w_{\varepsilon} \to u \quad \text{strongly in $\mathrm L^p_{\mathrm{loc}}(\R_+ \times \R)$, for every $1 \le p < \infty$},
\]
for some $u \in \mathrm L^\infty(\R_+ \times \R)$. 

Under the assumptions of \cref{it:exp}, we can apply \cref{prop:cc} to the family $\{u_\varepsilon\}_{\varepsilon >0}$ as well and deduce that (up to subsequences)
\[
w_{\varepsilon},\, u_{\varepsilon} \to u \quad \text{strongly in $\mathrm L^p_{\mathrm{loc}}(\R_+ \times \R)$, for every $1 \le p < \infty$}.
\]

By \cite[Theorem 1.2]{2206.03949}, the limit function $u$ is in fact the unique entropy-admissible solution of the limit problem \cref{eq:cl-l}.

As a consequence of the uniqueness of entropy solutions, we deduce (owing to Uryshon's subsequence principle) that the whole family (not just up to subsequences) converges.

This completes the proofs of \cref{it:conv-w} and \cref{it:conv-u}.
\end{proof}

\begin{remark}[Convergence of $\{u_\varepsilon\}_{\varepsilon>0}$ in the case $\gamma(\cdot) \coloneqq \mathds{1}_{]-\infty,0]}(\cdot)\exp(\cdot)$]\label{rk:conv-u}
In \cref{it:conv-u}, we relied on \cref{it:exp} to prove the convergence of $\{u_\varepsilon\}_{\varepsilon >0}$ to the (unique) entropy solution $u$ of the Cauchy problem \cref{eq:cl-l}. 

However, this convergence can also be established without using the additional assumption on $V$ needed in \cref{it:exp}, which was required there to prove the compactness of the entropy production of $\{u_\varepsilon\}_{\varepsilon>0}$. The alternative argument relies instead on \cref{it:conv-w}, the identity \cref{eq:exp-identity}, and the total variation estimate proved in \cite[Theorems~3.2 and~4.2]{MR4651679}. 

Since $u_0$ satisfies \cref{ass:u0}, the function $w_\varepsilon(0,\cdot) = \gamma_\eps \ast u_0$ fulfills
\[
\mathrm{TV}\big(w_\varepsilon(0,\cdot)\big)
\leq \varepsilon^{-1}\,\mathrm{TV}(\gamma)\, \|u_0\|_{\mathrm L^\infty(\R)},
\]
(i.\,e., its total variation is finite for a fixed $\varepsilon>0$, but blows up as $\varepsilon \searrow 0$).
By the total variation estimate in \cite[Theorems~3.2 and~4.2]{MR4651679}, we then have
\[
\mathrm{TV}\big(w_\varepsilon(t,\cdot)\big) \leq \mathrm{TV}\big(w_\varepsilon(0,\cdot)\big) .
\]
Multiplying both sides by $\varepsilon$ and using the relation $\varepsilon\,\partial_x w_\varepsilon = w_\varepsilon - u_\varepsilon$ yields
\begin{align}\label{eq:dist-est}
\begin{aligned}
\|w_\varepsilon(t,\cdot)-u_\varepsilon(t,\cdot)\|_{\mathrm{L}^1(\R)} 
&= \varepsilon\, \mathrm{TV}(w_\varepsilon(t,\cdot)) \\
&\le \varepsilon\, \mathrm{TV}(w_\varepsilon(0,\cdot)) 
\\ &= \|w_\varepsilon(0,\cdot)-u_\varepsilon(0,\cdot)\|_{\mathrm{L}^1(\R)}
\\ &= \|u_0 \ast \gamma_\varepsilon - u_0\|_{\mathrm{L}^1(\R)} 
\to 0 \quad \text{as } \varepsilon \searrow 0.
\end{aligned}
\end{align}
Since, by \cref{it:conv-w}, the family $\{w_\varepsilon\}_{\varepsilon >0}$ is already known to converge strongly in $\mathrm{L}^p_{\mathrm{loc}}(\R_+\times\R)$, for every $1 \le p < \infty$, to the (unique) entropy solution $u$ of \cref{eq:cl-l}, estimate \cref{eq:dist-est} shows that $\{u_\varepsilon\}_{\varepsilon >0}$ converges to the same limit, thus proving \cref{it:conv-u} without invoking \cref{it:exp}.

More generally, we address the question of deducing the convergence of the family $\{u_\varepsilon\}_{\varepsilon >0}$ from that of $\{w_\varepsilon\}_{\varepsilon >0}$ in \cite{DeNittiKuang2025V} using a Fourier approach.
\end{remark}

\section{Conclusions and future directions}
\label{sec:conclusion}

In this work, we relied on compensated compactness to address a long-standing open problem in the theory of nonlocal conservation laws: the localization singular limit in the presence of kernels that do not satisfy convexity assumptions and initial data with unbounded variation. In particular, our analysis covers kernels that are piecewise constant or strictly monotone. It remains open to study convergence results for kernel satisfying only \cref{ass:gamma}.

In forthcoming works, we will leverage these techniques to analyze the convergence of \textit{asymptotically compatible numerical schemes} (see, e.\,g., \cite{zbMATH05813035,zbMATH05533163} for applications of compensated compactness to the convergence analysis of finite-volumes schemes for local conservation laws).

We are also interested in quantifying compensated compactness to establish regularity results in the spirit of \cite{MR2605213,MR2730802}.

Finally, a promising direction for future research is to exploit Schonbek's $\mathrm L^p$ compensated compactness theory (see \cite{Sch}) to treat nonlocal models where $\mathrm L^\infty$-bounds fail (cf.~\cite{KEIMER2023}).

\vspace{0.5cm}

\section*{Acknowledgments}

G.~M.~Coclite is a member of the Gruppo Nazionale per l'Analisi Matematica, la Probabilità e le loro Applicazioni (GNAMPA) of the Istituto Nazionale di Alta Matematica (INdAM). He has been partially supported by the project funded under the National Recovery and Resilience Plan (NRRP), Mission~4, Component~2, Investment~1.4 (Call for tender No.~3138 of 16/12/2021) of the Italian Ministry of University and Research, funded by the European Union (NextGenerationEU Award No.~CN000023, Concession Decree No.~1033 of 17/06/2022) and adopted by the Italian Ministry of University and Research (CUP~D93C22000410001), Centro Nazionale per la Mobilit\`a Sostenibile. He has also been supported by the Italian Ministry of Education, University and Research under the programme ``Department of Excellence'' Legge~232/2016 (CUP~D93C23000100001), and by the Research Project of National Relevance ``Evolution problems involving interacting scales'' granted by the Italian Ministry of Education, University and Research (MUR--PRIN~2022, project code~2022M9BKBC, CUP~D53D23005880006).

N.~De~Nitti is a member of the Gruppo Nazionale per l'Analisi Matematica, la Probabilità e le loro Applicazioni (GNAMPA) of the Istituto Nazionale di Alta Matematica (INdAM).

K.~Huang was supported by a Direct Grant of Research (2024/25) from The Chinese University of Hong Kong.

We thank M.~Colombo and L.~Spinolo for their helpful and constructive comments.

\vspace{0.5cm}

\printbibliography

\vfill 

\end{document}